\documentclass[a4paper,leqno,10pt]{amsart}

\usepackage{epsf,graphicx,epsfig}
\usepackage{amscd}
\usepackage{amsmath,latexsym,amssymb,amsthm}
\usepackage[nospace,noadjust]{cite}
\usepackage{textcomp}
\usepackage{setspace,cite}
\usepackage{lscape,fancyhdr,fancybox}
\usepackage{stmaryrd}
\usepackage[all,cmtip]{xy}
\usepackage{tikz}
\usepackage{cancel}
\usetikzlibrary{shapes,arrows,decorations.markings}
\usepackage{graphicx}

\usepackage{color}
\usepackage{url}
\usepackage{enumerate}
\usepackage[mathscr]{euscript}
\usepackage{mathtools}

  \theoremstyle{plain}

\swapnumbers
    \newtheorem{thm}{Theorem}[section]
    \newtheorem{prop}[thm]{Proposition}

    \newtheorem{corollary}[thm]{Corollary}
    
    \newtheorem{subsec}[thm]{}
\theoremstyle{definition}
    \newtheorem{defn}[thm]{Definition}
        \newtheorem{remark}[thm]{Remark}
    \newtheorem{exam}[thm]{Example}

\theoremstyle{remark}

\title{}
\author{}
\date{}
\usepackage{amssymb}

\usepackage{hyperref}
\hypersetup{
	colorlinks,
	citecolor=blue,
	filecolor=black,
	linkcolor=blue,
	urlcolor=black
}

\begin{document}
\title{Cohomology and deformations of twisted Rota-Baxter operators and NS-algebras}
\author{Apurba Das}
\address{Department of Mathematics and Statistics,
Indian Institute of Technology, Kanpur 208016, Uttar Pradesh, India.}
\email{apurbadas348@gmail.com}


\begin{abstract}
The aim of this paper is twofold. In the first part, we consider twisted Rota-Baxter operators on associative algebras that were introduced by Uchino as a noncommutative analogue of twisted Poisson structures. We construct an $L_\infty$-algebra whose Maurer-Cartan elements are given by twisted Rota-Baxter operators. This leads to cohomology associated to a twisted Rota-Baxter operator. This cohomology can be seen as the Hochschild cohomology of a certain associative algebra with coefficients in a suitable bimodule. We study deformations of twisted Rota-Baxter operators by means of the above-defined cohomology. Application is given to Reynolds operators. In the second part, we consider NS-algebras of Leroux that are related to twisted Rota-Baxter operators in the same way dendriform algebras are related to Rota-Baxter operators. We define cohomology of NS-algebras using multiplicative operads and study their deformations in terms of the cohomology.
\end{abstract}

\maketitle

\tableofcontents






\section{Introduction}
The notion of Rota-Baxter operators first appeared in the fluctuation theory of probability in a paper of G. Baxter \cite{baxter}. After that, Rota-Baxter operators were found connections with combinatorics by G.-C. Rota \cite{rota} and P. Cartier \cite{cart}. A Rota-Baxter operator is an algebraic abstraction of the integral operator, i.e., inverse operator of an (invertible) derivation. In \cite{aguiar} Aguiar showed that Rota-Baxter operators produce Loday's dendriform algebras. They have also important applications in the Connes-Kreimer's algebraic approach of the renormalization in quantum field theory \cite{conn}. In the last twenty years, Rota-Baxter operators and related structures are widely studied in the literature. See \cite{guo-book} for details. Recently, deformations of Rota-Baxter operators on Lie algebras has been developed in \cite{tang}. The results have been generalized to associative algebras by the author in \cite{das-rota}.

\medskip

A generalization of a Rota-Baxter operator in the presence of a bimodule was introduced by K. Uchino \cite{uchino} by the name of generalized Rota-Baxter operator (also known as relative Rota-Baxter operator or $\mathcal{O}$-operator).  Let $A$ be an associative algebra and $M$ be an $A$-bimodule. A linear map $T: M \rightarrow A$ is said to be a generalized Rota-Baxter operator if $T$ satisfies
\begin{align*}
T(u) T(v) = T ( T(u) \cdot v + u \cdot T(v)),~ \text{ for } u, v \in M.
\end{align*} 
He showed that generalized Rota-Baxter operators can be seen as an operator analogue of Poisson structures. This led him to introduce a notion of $H$-twisted Rota-Baxter operators as an operator analogue of twisted Poisson structures of \v{S}evera and Weinstein \cite{sev-wein}. Here $H$ is a Hochschild $2$-cocycle of $A$ with coefficients in the bimodule $M$. A Reynolds operator \cite{zhang-gao-guo,reynolds} on $A$ is a $(- \mu)$-twisted Rota-Baxter operator, where $\mu$ defines the associative multiplication on $A$. Uchino also observed that an $H$-twisted Rota-Baxter operator induces an NS-algebra of Leroux \cite{leroux} in the same way a Rota-Baxter operator induces a dendriform algebra structure.

\medskip

Our aim in this paper is to study cohomology and deformations of $H$-twisted Rota-Baxter operators and NS-algebras. At first, using a graded Lie algebra constructed in \cite{das-rota} (whose Maurer-Cartan elements are Rota-Baxter operators), here we construct an $L_\infty$-algebra whose Maurer-Cartan elements are $H$-twisted Rota-Baxter operators. This characterization suggests us to define a cohomology theory to any $H$-twisted Rota-Baxter operator $T$, called the cohomology of $T$. (See the Appendix at the end for some basics on $L_\infty$-algebras and Maurer-Cartan elements). Next, we view this cohomology as the Hochschild cohomology of a certain algebra with coefficients in a suitable bimodule. More precisely, given an $H$-twisted Rota-Baxter operator $T: M \rightarrow A$, we know from \cite{uchino} that $M$ carries an associative algebra structure. We show that there is an $M$-bimodule structure on $A$. Further, we show that the corresponding Hochschild cohomology is isomorphic to the cohomology of $T$.

\medskip

The classical formal deformation theory of Gerstenhaber \cite{gers} has been applied to Rota-Baxter operators in \cite{tang,das-rota}. Here we further extend it to $H$-twisted Rota-Baxter operators. We show that the linear term in a deformation of an $H$-twisted Rota-Baxter operator $T$ is a $1$-cocycle in the cohomology of $T$. This is called the infinitesimal of the deformation. Moreover, equivalent deformation has cohomologous infinitesimals. We introduce Nijenhuis elements associated with an $H$-twisted Rota-Baxter operator $T$ and find a necessary condition for the rigidity of $T$ in terms of Nijenhuis elements.

\medskip

Next, we focus on NS-algebras. Given a vector space $A$, we first construct a new non-symmetric operad $\mathcal{O}_A$. A multiplication on this operad $\mathcal{O}_A$ is given by an NS-algebra structure on $A$. See \cite{gers-voro,das-loday} for more on non-symmetric operads and multiplications on them.  Thus, an NS-algebra $A$ gives rise to a cochain complex induced from the corresponding multiplicative operad following the approach of Gerstenhaber and Voronov \cite{gers-voro}. The cohomology of this cochain complex is called the cohomology of the NS-algebra $A$. As a consequence of the construction, one can conclude that the cohomology inherits a Gerstenhaber algebra structure in the sense of \cite{gers-ring}. We show that there is a morphism of non-symmetric operads $\mathcal{O}_A \rightarrow \mathrm{End}_A$, where $\mathrm{End}_A$ is the endomorphism operad associated to the vector space $A$. This, in particular, implies that there is a morphism from the cohomology of an NS-algebra $A$ to the Hochschild cohomology of the induced associative algebra structure on $A$. 
Finally, we describe deformations of an NS-algebra by means of the cohomology.

\medskip

The paper is organized as follows. In the next section (section \ref{sec-2}), we recall $H$-twisted Rota-Baxter operators, Reynolds operators and NS-algebras. In section \ref{sec-cohomo-trb}, we first give the Maurer-Cartan characterization of an $H$-twisted Rota-Baxter operator $T$  and using it, we define cohomology of $T$. We also view the cohomology of $T$ as the Hochschild cohomology of a suitable algebra with coefficients in a suitable bimodule.
Deformations of $H$-twisted Rota-Baxter operators are considered in section \ref{sec-def}. In section \ref{sec-ns-alg}, we define cohomology of NS-algebras using multiplicative non-symmetric operads and study deformations of NS-algebras. 

\medskip

All vector spaces, linear maps and tensor products are over a field $\mathbb{K}$ of characteristic $0$. 

\section{Twisted Rota-Baxter operators}\label{sec-2}
In this section, we recall $H$-twisted Rota-Baxter operators, Reynolds operators on associative algebras and NS-algebras \cite{uchino,leroux}. Meanwhile, we also give some constructions of twisted Rota-Baxter operators out of an old one.

Let $A$ be an associative algebra. We denote the multiplication map on $A$ by $\mu : A^{\otimes 2} \rightarrow A$ and write $\mu (a, b ) = ab$, for $a, b \in A$. An $A$-bimodule is a vector space $M$ together with bilinear maps $l : A \otimes M \rightarrow M,~(a, u ) \mapsto a \cdot u$ and $r: M \otimes A \rightarrow M,~(u, a) \mapsto u \cdot a$ (called left and right actions of $A$, respectively) satisfying 
\begin{align*}
(ab) \cdot u  = a \cdot (b \cdot u), \qquad (a \cdot u) \cdot b = a \cdot (u \cdot b), \qquad  (u \cdot a) \cdot b  =  u \cdot (ab ),
\end{align*}
for $a, b \in A$ and $u \in M$. It follows that $A$ is an $A$-bimodule where the left and right actions of $A$ are given by the algebra multiplication. 

The Hochschild cohomology of $A$ with coefficients in the $A$-bimodule $M$ is given by the cohomology of the cochain complex $\{ C^\bullet_{\mathrm{Hoch}}(A, M), \delta_{\mathrm{Hoch}} \}$, where $C^n_{\mathrm{Hoch}}(A, M) = \mathrm{Hom}(A^{\otimes n}, M)$, for $n \geq 0$ and
\begin{align*}
(\delta_{\mathrm{Hoch}} f) (a_1, \ldots, a_{n+1}) =~& a_1 \cdot f(a_2, \ldots, a_{n+1}) 
+ \sum_{i=1}^n (-1)^i f (a_1, \ldots, a_{i-1}, a_i a_{i+1}, \ldots, a_{n+1}) \\
~&+ (-1)^{n+1} f(a_1, \ldots, a_n) \cdot a_{n+1}, \text{ for } f \in C^n_{\mathrm{Hoch}}(A, M).
\end{align*}

Let $H\in C^2_{\mathrm{Hoch}}(A, M)$ be a Hochschild $2$-cocycle, i.e., $H : A^{\otimes 2} \rightarrow M$ satisfies
\begin{align*}
a \cdot H(b,c) - H(ab, c) + H (a, bc) - H(a,b) \cdot c = 0,~ \text{ for all } a, b, c \in A.
\end{align*}

\begin{remark}
Note that a $2$-cocycle $H$ induces an associative algebra structure on $A \oplus M$, called the $H$-twisted semi-direct product (denoted by $A \ltimes_H M$) whose multiplication is given by
\begin{align*}
(a,u) \cdot_H (b, v) = (ab, a \cdot v + u \cdot b + H(a, b)),~ \text{ for } (a,u), (b, v) \in A \oplus M.
\end{align*}
\end{remark}

\begin{defn}\cite{uchino}
A linear map $T : M \rightarrow A$ is said to be an {\bf $H$-twisted Rota-Baxter operator} if $T$ satisfies
\begin{align*}
T(u) T(v) = T(u \cdot T(v) + T(u) \cdot v + H (Tu, Tv)),~ \text{ for } u, v \in M.
\end{align*}
\end{defn}


The following result is straightforward.

\begin{prop}\label{graph-tw}
A linear map $T: M \rightarrow A$ is an $H$-twisted Rota-Baxter operator if and only if the graph of $T$,
\begin{align*}
\mathrm{Gr}(T) = \{ (Tu, u)| u \in M \}
\end{align*}
is a subalgebra of the $H$-twisted semi-direct product $A \ltimes_H M$.
\end{prop}

\begin{remark}
Note that a Poisson manifold is a smooth manifold $M$ equipped with a bivector field $\pi \in \Gamma(\wedge^2 TM)$ which induces a Poisson algebra structure on $C^\infty(M)$. A bivector field $\pi$ is a Poisson structure on $M$ if and only if the graph $\mathrm{Gr}(\pi^\sharp) \subset TM \oplus T^*M$ of the induced bundle map $\pi^\sharp : T^*M \rightarrow TM,~ \alpha \mapsto \pi ( \alpha, -)$ is closed under the Courant bracket $[~,~]_{\mathrm{Cou}}$ on $TM \oplus T^*M$ \cite{courant}. Rota-Baxter operators can be seen as associative analogue of Poisson structures \cite{uchino}. In \cite{sev-wein} \v{S}evera and Weinstein introduced a notion of twisted Poisson structure in the study of Poisson geometry modified by a closed $3$-form. See \cite{klimcik-strobl} for the geometric aspects of twisted Poisson structures. Given a closed $3$-form $H \in \Omega^3_{\mathrm{cl}}(M)$ on a manifold $M$, they considered an $H$-twisted Courant bracket $[~,~]^H_{\mathrm{Cou}}$ on $TM \oplus T^*M$. A bivector field $\pi$ is called an $H$-twisted Poisson structure if $\mathrm{Gr}(\pi^\sharp)$ is closed under $[~,~]^H_{\mathrm{Cou}}$. In this regard, $H$-twisted Rota-Baxter operators are an associative analogue of twisted Poisson structures.
\end{remark}

\begin{exam}
Any Rota-Baxter operator of weight $0$ (more generally any relative Rota-Baxter operator) is an $H$-twisted Rota-Baxter operator with $H = 0 $.
\end{exam}

\begin{exam} \cite{uchino}
Let $A$ be an associative algebra, $M$ be an $A$-bimodule and $h : A \rightarrow M$ be an invertible Hochschild $1$-cochain. Then $T= h^{-1} : M \rightarrow A$ is an $H$-twisted Rota-Baxter operator with $H = - \delta_{\mathrm{Hoch}} h.$
\end{exam}

\begin{exam}\cite{uchino} Let $A$ be an associative algebra with multiplication map $\mu: A \otimes A \rightarrow A$. Note that the space $M = A \otimes A$ can be given an $A$-bimodule structure by $a \cdot (b \otimes c) = ab \otimes c$ and $(b \otimes c) \cdot a = b \otimes ca$, for $a \in A, ~ b \otimes c \in M=  A \otimes A$. Moreover, the map $H : A \otimes A \rightarrow M, ~ ( a \otimes b) \mapsto - a \otimes b$ is a Hochschild $2$-cocycle in the cohomology of $A$ with coefficients in the $A$-bimodule $M$. Then the multiplication map $\mu : M \rightarrow A$ is an $H$-twisted Rota-Baxter operator.
\end{exam}

\begin{exam}
Let $N : A \rightarrow A$ be a Nijenhuis operator on an associative algebra $A$ \cite{grab}, i.e., $N$ satisfies
\begin{align*}
N(a) N(b) = N (a N(b) + N(a) b - N (ab)),~ \text{ for } a, b \in A.
\end{align*}
Then $A$ carries a new associative algebra structure with product $a \cdot_N b = a N(b) + N(a) b - N (ab) ,$ for $a, b \in A$. We denote this algebra by $A_N$. Moreover, the vector space $A$ has an $A_N$-bimodule structure given by $a \cdot b = N(a) b $ and $b \cdot a = b N(a)$, for $a \in A_N,~ b \in A$. With this notation, the map $H : A_N \otimes A_N \rightarrow A,~ H (a, b) = - N(ab)$ is a Hochschild $2$-cocycle in the cohomology of $A_N$ with coefficients in the $A_N$-bimodule $A$. It is easy to see that the identity map $\mathrm{id} A \rightarrow A_N$ is an $H$-twisted Rota-Baxter operator. 
\end{exam}

\begin{prop}\cite{uchino}\label{prop-h-ass}
Let $T : M \rightarrow A$ be an $H$-twisted Rota-Baxter operator. Then $M$ carries an associative algebra structure with product given by
\begin{align*}
u * v = u \cdot T(v) + T(u) \cdot v + H (Tu, Tv), ~~ \text{ for } u, v \in M.
\end{align*}
\end{prop}

\medskip

In the following, we construct new twisted Rota-Baxter operators out an old one suitably modified by Hochschild cocycles. We start with the following result. 

\begin{prop}
Let $A$ be an associative algebra, $M$ be an $A$-bimodule and $H$ be a Hochschild $2$-cocycle. Then for any Hochschild $1$-cochain $h$, there is an isomorphism of twisted semi-direct products
\begin{align*}
A \ltimes_H M  ~\simeq~ A \ltimes_{H + \delta h} M.
\end{align*}
\end{prop}

\begin{proof}
We define a map $\Phi_h : A \ltimes_H M  \mapsto A \ltimes_{H + \delta h} M$ by
$\Phi_h (a, u) = (a, u - h(a)),$ for $(a, u) \in A \oplus M$. Then we have
\begin{align*}
\Phi_h ((a,u) \cdot_H (b,v)) =~& \Phi_h (ab, a \cdot v +  u \cdot b + H(a,b)) \\
=~&(ab, a \cdot v+u \cdot b + H(a,b) - h (ab)) \\
=~& (ab, a \cdot v + u \cdot b + H(a,b) - a \cdot h(b) - h(a) \cdot b + (\delta h)(a,b) ) \\
=~& (a , u - h(a)) \cdot_{H + \delta h} (b, v - h (b)) \\
=~& \Phi_h (a,u) \cdot_{H + \delta h} \Phi_h (b,v).
\end{align*}
Hence the proof.
\end{proof}

Let $T: M \rightarrow A$ be an $H$-twisted Rota-Baxter operator. Consider the subalgebra $\mathrm{Gr} (T) \subset A \ltimes_H M$ of the twisted semi-direct product. For any Hochschild $1$-cochain $h$, the image $$\Phi_h ( \mathrm{Gr}(T)) = \{ ( Tu, u - h (T u)) |~ u \in M \} \subset A \ltimes_{H + \delta h} M$$ is a subalgebra. However, this subalgebra may not be the graph of a linear map from $M$ to $A$. If the linear map $(\mathrm{id} - h \circ T) : M \rightarrow M$ is invertible, then $\Phi_h ( \mathrm{Gr}(T))$ is the graph of the map $T ( \mathrm{id} - h \circ T)^{-1}$. In this case, $T ( \mathrm{id} - h \circ T)^{-1}$ is an $(H + \delta h)$-twisted Rota-Baxter operator.


\medskip

Next, we perturb an $H$-twisted Rota-Baxter operator by a suitable Hochschild $1$-cocycle motivated by the gauge transformations of Poisson structures \cite{sev-wein}. Let $T: M \rightarrow A$ be an $H$-twisted Rota-Baxter operator. Consider the graph $\mathrm{Gr}(T) \subset A \ltimes_H M$ which is a subalgebra of the twisted semi-direct product. For any Hochschild $1$-cocycle $B$, consider the deformed subspace
\begin{align*}
\tau_B (\mathrm{Gr}(T)) := \{ (Tu, u + B(Tu)) |~ u \in M \}.
\end{align*}
Then $\tau_B (\mathrm{Gr}(T)) \subset A \ltimes_H M$ is a subalgebra as
\begin{align*}
&( Tu, u + B(Tu)) \cdot_H ( Tv , v + B (Tv)) \\
&= ( T(u)T(v), (Tu)  \cdot v + (Tu) \cdot B(Tv) + u \cdot (Tv) + B (Tu) \cdot Tv + H (Tu, Tv)) \\
&= ( T(u )T(v), (Tu) \cdot v + u \cdot (Tv) + B (Tu, Tv) + H (Tu, Tv)) \in \tau_B (\mathrm{Gr}(T)).
\end{align*}
If the bundle map $\mathrm{id} + B \circ T : M \rightarrow M$ is invertible then $\tau_B (\mathrm{Gr}(T)) \subset A \ltimes_H M$ is the graph of the linear map $T (\mathrm{id} + B \circ T )^{-1}$. In this case, the $1$-cocycle $B$ is called $T$-admissible. Then the linear map $T (\mathrm{id} + B \circ T )^{-1}$ is an $H$-twisted Rota-Baxter operator by Proposition \ref{graph-tw}. We call this $H$-twisted Rota-Baxter operator as gauge transformation of $T$, denoted by $T_B$.

\begin{prop}
Let $T$ be an $H$-twisted Rota-Baxter operator and $B$ be a $T$-admissible $1$-cocycle. Then the associative algebra structures on $M$ induced from $H$-twisted Rota-Baxter operators $T$ and $T_B$ are isomorphic.
\end{prop}

\begin{proof}
Consider the vector space isomorphism $\mathrm{id} + B \circ T : M \rightarrow M$. We have
\begin{align*}
&(\mathrm{id} + B \circ T) (u) *_B (\mathrm{id} + B \circ T) (v) \\
&= (Tu) (\mathrm{id} + B \circ T)(v) + (\mathrm{id} + B \circ T) (u) (Tv) + H (Tu, Tv) \\
&= T(u) \cdot v + u \cdot T(v) + (Tu) \cdot B (Tv) + B (Tu) \cdot (Tv) + H (Tu, Tv) \\
&= u * v + B ( (Tu) (Tv)) \\
&= u * v + BT (u * v) = (\mathrm{id} + B \circ T)(u * v).
\end{align*}
Hence the proof.
\end{proof}

Next, we recall Reynolds operators on associative algebras which are special cases of twisted Rota-Baxter operators \cite{reynolds} (see also \cite{zhang-gao-guo}).

\begin{defn}
Let $A$ be an associative algebra. A linear map $R: A \rightarrow A$ is called a {\bf Reynolds operator} if $R$ satisfies
\begin{align*}
R(a) R(b) = R ( aR(b) + R(a) b - R(a) R(b)),~ \text{ for } a, b \in A.
\end{align*}
\end{defn}

Note that, if $\mu : A^{\otimes 2} \rightarrow A$ denotes the associative multiplication on $A$ then $\mu$ is a Hochschild $2$-cocycle in the cohomology of $A$ with coefficients in itself. Thus, a Reynolds operator on $A$ is a $(-\mu)$-twisted Rota-Baxter operator.

Twisted Rota-Baxter operators (hence Reynolds operators) are related to NS-algebras in the same way Rota-Baxter operators are related to dendriform algebras \cite{aguiar,uchino}. 

\begin{defn} \cite{leroux}
An {\bf NS-algebra} is a vector space $A$ together with three binary operations $\prec, \succ, \curlyvee : A \otimes A \rightarrow A$ satisfying the following set of identities
\begin{align}
(a \prec b) \prec c =~& a \prec (b * c), \label{ns-1}\\
(a \succ b ) \prec c =~& a \succ (b \prec c), \label{ns-2}\\
(a * b ) \succ c =~& a \succ ( b \succ c), \label{ns-3}\\
(a \curlyvee b ) \prec c + (a * b ) \curlyvee c =~& a \succ (b \curlyvee c) + a \curlyvee (b * c), \label{ns-4}
\end{align}
for all $a, b, c \in A$, where $a *b = a \prec b + a \succ b + a \curlyvee b$.
\end{defn}

It follows from (\ref{ns-1}) - (\ref{ns-4}) that the operation $*$ is an associative multiplication on $A$. Therefore, NS-algebras are certain splitting of associative algebras by three operations. Note that, in an NS-algebra $(A, \prec, \succ, \curlyvee)$, if the binary operation $\curlyvee$ is trivial, then one reduces to a dendriform algebra \cite{loday}.

Let $(A, \prec, \succ, \curlyvee)$ and $(A', \prec', \succ', \curlyvee')$ be two NS-algebras. A morphism between them is a linear map $f : A \rightarrow A'$ satisfying $f (a \prec b) = f(a) \prec' f(b)$, $f ( a \succ b) = f(a) \succ' f( b)$ and $f (a \curlyvee b) = f(a) \curlyvee' f(b)$, for $a, b \in A$. NS-algebras and morphisms between them forms a category, denoted by {\bf NS}.

NS-algebras first appeared from Nijenhuis operators. More precisely, if $N : A \rightarrow A$ is a Nijenhuis operator on an associative algebra $A$, then the binary operations
\begin{align*}
a \prec b = a N(b), \quad a \succ b = N(a) b ~~~ \text{ and } ~~~ a \curlyvee b = - N (ab)
\end{align*}
makes $A$ into an NS-algebra. We remark that a Nijenhuis operator induces a more general notion of $N$-dendriform algebra introduced in \cite{lei-guo}. 

In the following, we recall from \cite[Proposition 3.3]{uchino} that  a 
twisted Rota-Baxter operator also gives rise to an NS-algebra.

\begin{prop}\label{trb-ns}
Let $T: M \rightarrow A$ be an $H$-twisted Rota-Baxter operator. Then $M$ carries an NS-algebra structure with binary operations 
\begin{align*}
u \prec v = u \cdot T(v), \quad u \succ v = T(u ) \cdot v ~~~ \text{ and } ~~~ u \curlyvee v = H (Tu, Tv), ~ \text{ for } u, v \in M. 
\end{align*}
\end{prop}

Note that the corresponding associative algebra structure on $M$ coincide with the one given in Proposition \ref{prop-h-ass}. More details study of NS-algebras from cohomological perspectives are made in section \ref{sec-ns-alg}.

\medskip

Let $A$ be an associative algebra, $M$ be an $A$-bimodule, $H \in C^2_{\mathrm{Hoch}}(A, M)$ be a Hochschild $2$-cocycle and $T: M \rightarrow A$ be an $H$-twisted Rota-Baxter operator. Suppose $A'$ is another associative algebra, $M'$ is an $A'$-bimodule, $H' \in C^2_{\mathrm{Hoch}}(A', M')$ is a Hochschild $2$-cocycle and $T': M' \rightarrow A'$ is an $H'$-twisted Rota-Baxter operator. 

\begin{defn}
A morphism of twisted Rota-Baxter operators from $T$ to $T'$ is a pair $(\phi, \psi)$ consisting of an algebra morphism $\phi : A \rightarrow A'$ and a linear map $\psi : M \rightarrow M'$ satisfying
\begin{align}
&\psi (a \cdot u) = \phi(a) \cdot  \psi(u) ~~~ \text{ and } ~~~ \psi (u \cdot a) = \psi(u) \cdot \phi(a), \label{action-comp}\\
&\psi \circ H = H' \circ ( \phi \otimes \phi), \label{h-comp}\\
&\phi \circ T = T' \circ \psi, ~ \text{ for } a \in A, u \in M. \label{t-comp}
\end{align}
\end{defn}

Twisted Rota-Baxter operators and morphisms between them forms a category, denoted by {\bf TRB}. Note that the construction of Proposition \ref{trb-ns} yields a functor ${\bf TRB} \rightarrow {\bf NS}$.

\section{Cohomology of twisted Rota-Baxter operators}\label{sec-cohomo-trb}
In \cite{das-rota} the author constructs a graded Lie algebra using Voronov's derived bracket \cite{voro} whose Maurer-Cartan elements are precisely Rota-Baxter operators. In this section, we include a ternary map on the underlying graded vector space of the graded Lie algebra to make it an $L_\infty$-algebra. This ternary bracket is made by using a Hochschild $2$-cocycle $H$. The Maurer-Cartan elements of this new $L_\infty$-algebra are precisely $H$-twisted Rota-Baxter operators. This characterization allows us to define cohomology for an $H$-twisted Rota-Baxter operator $T$. (See the Appendix for some basics on $L_\infty$-algebras and Maurer-Cartan elements). Finally, we show that the cohomology of $T$ can be seen as the Hochschild cohomology of $(M, *)$ with coefficients in a suitable bimodule structure on $A$.

\subsection{Maurer-Cartan characterization and cohomology} Let $A$ be an associative algebra and $M$ be an $A$-bimodule. Then it has been shown in \cite{das-rota} that the graded vector space $\bigoplus_{n \geq 0} \mathrm{Hom}( M^{\otimes n }, A)$ equipped with the bracket
\begin{align}\label{derived-brkt}
\llbracket P, Q \rrbracket = (-1)^p [[ \mu+ l + r , P]_{\mathsf{G}}, Q]_{\mathsf{G}},
\end{align}
for $P \in \mathrm{Hom}(M^{\otimes p}, A),~ Q \in \mathrm{Hom}(M^{\otimes q}, A)$, forms a graded Lie algebra. Here $[~,~]_{\mathsf{G}}$ is the classical Gerstenhaber bracket of multilinear maps on the direct sum vector space $A \oplus M$. The explicit description of the bracket $\llbracket ~,~ \rrbracket$ can be found in the above mentioned reference. It has been observed that for a linear map $T \in \mathrm{Hom}(M, A)$,
\begin{align}\label{tt}
\llbracket T, T \rrbracket (u, v) = 2 \big(  T ( T(u) \cdot v + u \cdot T(v)) - T(u) T(v) \big),~ \text{ for } u, v \in M.
\end{align}
The graded Lie bracket $\llbracket ~, ~ \rrbracket$ is the associative analog of the classical Schouten-Nijenhuis bracket of multivector fields on a manifold $M$. Given a closed $3$-form on $M$, in \cite{fre-zam}, the authors constructs a ternary bracket on the space of multivector fields which together with the Schouten-Nijenhuis bracket makes the graded space of multivector fields an $L_\infty$-algebra. The Maurer-Cartan elements of this $L_\infty$-algebra are twisted Poisson structures. The associative analog can be rephrased as follows.

Let $H \in C^2_{\mathrm{Hoch}}(A, M)$ be a Hochschild $2$-cocycle of $A$ with coefficients in $M$. Define a degree $-1$ ternary bracket $\llbracket ~, ~ \rrbracket$ on the graded vector space $\bigoplus_{n \geq 0} \mathrm{Hom}( M^{\otimes n }, A)$ by
\begin{align}\label{ter-brk}
&\llbracket P, Q, R \rrbracket (u_1, \ldots, u_{p+q+r-1}) \\
&= (-1)^{pqr} \bigg\{ \sum_{1 \leq i \leq p} (-1)^{(i-1) q}~ P \big( u_1, \ldots, u_{i-1}, H \big( Q (u_i, \ldots, u_{i+q-1}), R(u_{i+q}, \ldots, u_{i+q+r-1}) \big), \ldots, u_{p+q+r-1} \big)  \nonumber \\
&- (-1)^{qr} \sum_{1 \leq i \leq p} (-1)^{(i-1) r}~ P \big( u_1, \ldots, u_{i-1}, H \big( R (u_i, \ldots, u_{i+r-1}), Q(u_{i+r}, \ldots, u_{i+q+r-1}) \big), \ldots, u_{p+q+r-1} \big)  \nonumber \\
&- (-1)^{pq} \sum_{1 \leq i \leq q} (-1)^{(i-1) p}~ Q \big( u_1, \ldots, u_{i-1}, H \big( P (u_i, \ldots, u_{i+p-1}), R(u_{i+p}, \ldots, u_{i+p+r-1}) \big), \ldots, u_{p+q+r-1} \big)  \nonumber  \\
&+ (-1)^{p(q+r)} \sum_{1 \leq i \leq q} (-1)^{(i-1) r}~ Q \big( u_1, \ldots, u_{i-1}, H \big( R (u_i, \ldots, u_{i+r-1}), P(u_{i+r}, \ldots, u_{i+p+r-1}) \big), \ldots, u_{p+q+r-1} \big)  \nonumber \\
&- (-1)^{pq+qr+rp} \sum_{1 \leq i \leq r} (-1)^{(i-1) q}~ R \big( u_1, \ldots, u_{i-1}, H \big( Q (u_i, \ldots, u_{i+q-1}), P(u_{i+q}, \ldots, u_{i+p+q-1}) \big), \ldots, u_{p+q+r-1} \big)  \nonumber \\
&+ (-1)^{(p+q)r}  \sum_{1 \leq i \leq r} (-1)^{(i-1) p}~ R \big( u_1, \ldots, u_{i-1}, H \big( P (u_i, \ldots, u_{i+p-1}), Q(u_{i+p}, \ldots, u_{i+p+q-1}) \big), \ldots, u_{p+q+r-1} \big) \bigg\}.  \nonumber 
\end{align}
It is easy to see that the ternary bracket $\llbracket ~, ~, ~ \rrbracket$ is graded skew-symmetric. Similar to the proof of \cite{fre-zam}, one can also show that the binary bracket $\llbracket ~, ~\rrbracket$ and the ternary bracket $\llbracket ~, ~ , ~ \rrbracket$ are compatible in the sense of an $L_\infty$-algebra (with higher maps are trivial). Moreover, we have the following.

\begin{thm}\label{mc-main}
Let $A$ be an associative algebra, $M$ be an $A$-bimodule and $H$ be a Hochschild $2$-cocycle. A linear map $T: M \rightarrow A$ is an $H$-twisted Rota-Baxter operator if and only if $T \in \mathrm{Hom}(M, A)$ is a Maurer-Cartan element in the $L_\infty$-algebra $ (  \bigoplus_{n \geq 0} \mathrm{Hom}(M^{\otimes n}, A), \llbracket ~, ~ \rrbracket, \llbracket ~, ~, ~ \rrbracket)$.
\end{thm}

\begin{proof}
For a linear map $T: M \rightarrow A$, we have from (\ref{ter-brk}) that
\begin{align}\label{ttt}
\llbracket T, T, T \rrbracket (u, v) = - 6 T ( H (Tu, Tv)),~ \text{ for } u, v \in M.
\end{align}
Thus, from (\ref{tt}) and (\ref{ttt}), we get
\begin{align*}
\big( \frac{1}{2} \llbracket T, T \rrbracket - \frac{1}{6} \llbracket T, T, T \rrbracket \big) (u, v) =   T (T(u) \cdot v +  u \cdot T(v)) - T(u)T(v) + T (H (Tu, Tv)).
\end{align*}
Therefore, $T$ is a Maurer-Cartan element if and only if $T$ is an $H$-twisted Rota-Baxter operator.
\end{proof}

Let $T$ be an $H$-twisted Rota-Baxter operator. It follows from the above proposition that $T$ induces a differential $d_T : \mathrm{Hom}(M^{\otimes n}, A) \rightarrow \mathrm{Hom}(M^{\otimes n}, A)$, for $n \geq 0$, by
\begin{align*}
d_T (f) = \llbracket T, f \rrbracket - \frac{1}{2} \llbracket T, T, f \rrbracket, ~ \text{ for } f \in \mathrm{Hom}(M^{\otimes n}, A).
\end{align*}
It follows from the definition of the brackets $\llbracket ~, ~ \rrbracket$ and $\llbracket ~,~, ~ \rrbracket$ that a linear map $f \in \mathrm{Hom}(M, A)$ satisfies $d_T (f) = 0$ if
\begin{align}\label{dt-1-co}
T ( f(u) \cdot v ~+~ &u \cdot f(v)) + f ( T(u) \cdot v + u \cdot T(v)) - T(u) f(v) - f(u) T(v) \\
& + T \big( H ( T(u), f(v)) + H ( f(u), T(v)) \big) + f ( H (T(u), T(v)) ) = 0,~ \text{ for } u, v \in M. \nonumber
\end{align}

For each $n \geq 0$, we define
\begin{align*}
\mathcal{Z}^n_T (M, A) = \{ f \in \mathrm{Hom}(M^{\otimes n}, A) |~ d_T (f) = 0 \} ~~~~ \text{ and } ~~~~ 
\mathcal{B}^n_T (M, A) = \{ d_T g |~ g \in \mathrm{Hom}(M^{\otimes n-1}, A) \}
\end{align*}
be the space of $n$-cocycles and $n$-coboundaries, respectively. The corresponding cohomology groups 
$$\mathcal{H}^n_T (M, A) = \frac{\mathcal{Z}^n_T (M, A)}{\mathcal{B}^n_T (M, A)}, ~ \text{ for } n \geq 0$$
are called the cohomology of the $H$-twisted Rota-Baxter operator $T$.

Given an $L_\infty$-algebra and a Maurer-Cartan element of it, one can construct a new $L_\infty$-algebra twisted by the Maurer-Cartan element \cite{getzler,markl}. In the present context, this simply gives the following.

\begin{thm}
Let $T : M \rightarrow A$ be an $H$-twisted Rota-Baxter operator. Then the graded vector space $\bigoplus_{n \geq 0} \mathrm{Hom}(M^{\otimes n}, A)$ carries a new $L_\infty$-algebra (called the twisted $L_\infty$-algebra) with structure maps
\begin{align*}
l_1 (P ) = d_T (P), ~~ \quad l_2 (P, Q) = \llbracket P, Q \rrbracket - \llbracket T, P, Q \rrbracket, ~~ \quad l_3 (P, Q, R) = \llbracket P, Q, R \rrbracket
\end{align*}
and trivial higher maps. Moreover, for any linear map $T' : M \rightarrow A$, the sum $T+ T'$ is an $H$-twisted Rota-Baxter operator if and only if $T'$ is a Maurer-Cartan element in the twisted $L_\infty$-algebra.
\end{thm}

\begin{proof}
The first part from the standard construction of twisted $L_\infty$-algebra \cite{getzler,markl}. For the second part, we observe that
\begin{align*}
&\frac{1}{2} \llbracket T + T' , T + T' \rrbracket - \frac{1}{6} \llbracket T + T', T + T', T + T' \rrbracket \\
&= \frac{1}{2} \llbracket T, T' \rrbracket + \frac{1}{2} \llbracket T', T \rrbracket + \frac{1}{2} \llbracket T', T' \rrbracket  - \frac{1}{6} \llbracket T, T, T' \rrbracket -  \frac{1}{6} \llbracket T, T', T \rrbracket -  \frac{1}{6} \llbracket T', T, T \rrbracket \\
&- \frac{1}{6} \llbracket T, T', T' \rrbracket -  \frac{1}{6} \llbracket T', T, T' \rrbracket -  \frac{1}{6} \llbracket T', T', T \rrbracket - \frac{1}{6} \llbracket T', T', T' \rrbracket \\
&= (\llbracket T, T' \rrbracket - \frac{1}{2} \llbracket T, T, T' \rrbracket)   + \frac{1}{2} ( \llbracket T', T' \rrbracket - \llbracket T, T', T' \rrbracket )  - \frac{1}{6} \llbracket T', T', T' \rrbracket \\
&= l_1 (T') + \frac{1}{2}~ l_2( T', T')  - \frac{1}{6} ~l_3 (T', T', T').
\end{align*}
This shows that $T+ T'$ is an $H$-twisted Rota-Baxter operator (i.e., the left hand side vanishes) if and only if $T'$ is a Maurer-Cartan element in the twisted $L_\infty$-algebra (i.e., the right hand side vanishes). Hence the proof.
\end{proof}




\subsection{Cohomology of twisted Rota-Baxter operator as Hochschild cohomology}

Let $T : M \rightarrow A$ be an $H$-twisted Rota-Baxter operator. Then it follows from Proposition \ref{prop-h-ass} that $M$ carries an associative algebra structure with product
\begin{align*}
u * v = u \cdot T(v) + T(u) \cdot v + H (Tu, Tv),~ \text{ for } u, v \in M.
\end{align*}
Moreover, we have the following.
\begin{prop}
The maps $l_T : M \otimes A \rightarrow A$ and $r_T : A \otimes M \rightarrow A$ given by
\begin{align*}
l_T (u, a) = T(u) a - T (u \cdot a + H (Tu, a)),  \qquad
r_T(a, u) = a T(u) - T(a \cdot u + H(a, Tu)), ~~~  u \in M, a \in A
\end{align*}
defines an $(M , *)$-bimodule structure on $A$.
\end{prop}

\begin{proof}
For any $u, v \in M$ and $a \in A$, we have
\begin{align*}
&l_T ( u * v, a) - l_T (u, l_T (v, a)) \\
&= T(u * v) a - T( (u *v) \cdot a) - TH (T(u * v), a) - l_T (u, T(v) a- T(v \cdot a) - TH (Tv, a))\\
&= \cancel{T(u) T(v) a} - \bcancel{T ( (u \cdot T(v)) \cdot a)} - T ( (T(u)\cdot v)\cdot a) - T (H (Tu, Tv)\cdot  a) - TH ( T(u) T(v), a) \\
&- \cancel{T(u) T(v) a} + T(u) T(v \cdot a) + T(u) TH(Tv,a) + \bcancel{T ( u \cdot (T(v) a))} - T (u \cdot T(v \cdot a)) - u \cdot TH(Tv, a)\\
&+ TH (Tu, T(v) a) - TH( Tu, T(v \cdot a)) - TH (Tu, TH(Tv, a)) \\
&= - T ( (T(u) \cdot v) \cdot a) - T (H (Tu, Tv)\cdot a) - TH ( T(u) T(v), a) + T (u \cdot T(v \cdot a)) + T ( T(u) \cdot (v \cdot a))\\ &+ TH (Tu, T(v \cdot a)) + T (u \cdot TH (Tv,a)) + T ( T(u) \cdot H(Tv, a)) + TH (Tu, TH(Tv,a)) \\& - T (u \cdot T(v \cdot a)) - T(u \cdot TH(Tv, a)) + TH (Tu, T(v) a) - TH( Tu, T(v \cdot a)) - TH (Tu, TH(Tv, a)) \\
&\stackrel{\text{after cancellations}}{=} - T (H (Tu, Tv)\cdot a) - TH ( T(u) T(v), a) + T ( T(u) \cdot H(Tv, a)) + TH (Tu, T(v) a)\\
&= T((\delta_{\mathrm{Hoch}} H)(Tu, Tv, a)) = 0.
\end{align*}
Similarly, we observe that
\begin{align*}
&r_T (l_T (u, a), v) - l_T (u, r_T (a, u)) \\
&= r_T (T(u)a - T(u \cdot a) -TH(Tu, a), v) - l_T (u, a T(v) - T(a \cdot v) -TH(a, Tv)) \\
&= \cancel{T(u) a T(v)} -T(u \cdot a) T(v) - TH(Tu,a ) T(v) - T( (T(u)a) \cdot v) + T (T(u \cdot a) \cdot v)  \\
&+ T (TH(Tu, a) \cdot v) - TH (T(u)a , Tv ) + TH (T (u \cdot a), Tv) + TH ( TH (Tu, a), Tv) \\
&- \cancel{T(u) a T(v)} + T(u) T(a \cdot v) + T(u) TH(a, Tv) + T (u \cdot (a T(v))) - T(u \cdot T(a \cdot v)) \\
&- T( u \cdot TH(a, Tv)) + TH (Tu, a T(v)) - TH (Tu, T(a \cdot v)) - TH (Tu, TH (a, Tv))\\
&= - T ((u \cdot a) \cdot T(v)) - T (T(u \cdot a) \cdot v) - TH (T(u \cdot a), Tv) - T (H(Tu, a) \cdot Tv) - T ( TH (Tu, a) \cdot v) \\
&- TH (TH(Tu, a), Tv) - T (((Tu)a) \cdot v) + T (T(u \cdot a) \cdot v) + T (TH(Tu, a) \cdot v) - TH (T(u) a, Tv)\\
& + TH ( T(u \cdot a), Tv) + TH ( TH (Tu, a), Tv) + T (u \cdot T(a \cdot v)) + T (T(u) \cdot (a \cdot v)) + TH (Tu, T(a \cdot v)) \\
&+ T (u \cdot TH(a, Tv)) + T (T(u) \cdot H (a, Tv)) + TH (Tu, TH (a, Tv)) + T (u \cdot (aT(v))) -  T(u \cdot  T(a \cdot v)) \\
& - T ( u \cdot TH(a, Tv)) + TH (Tu, a T(v)) - TH (Tu, T(a \cdot v)) - TH (Tu, TH(a, Tv))\\
& \stackrel{\text{after cancellations}}{=}  - T (H(Tu, a) \cdot Tv) - TH (T(u) a, Tv) + T (T(u) \cdot H (a, Tv)) + TH (Tu, a T(v))\\
&= T((\delta_{\mathrm{Hoch}} H)(Tu, a, Tv)) =  0,
\end{align*}
and
\begin{align*}
&r_T ( r_T (a, u), v) - r_T ( a, u * v) \\
&= r_T ( aT(u) - T(a \cdot u) -TH(a, Tu), v) - a T (u * v) + T (a \cdot (u*v)) + TH (a, T(u*v)) \\
&= \cancel{(a T(u)) T(v)} - T(a \cdot u) T(v) - TH (a, Tu) T(v) - T ((a T(u)) \cdot v) + T ( T(a \cdot u) \cdot v) + T ( TH (a, Tu) \cdot v)\\
&- TH ( a T(u), Tv) + TH ( T(a \cdot u), Tv) + TH ( TH (a, Tu), Tv)  - \cancel{a T(u) T(v)}  \\
&+ T ( a \cdot (u \cdot T(v))) + T (a \cdot ((Tu) \cdot v)) + T ( a \cdot H (Tu, Tv)) + TH (a, T(u) T(v)) \\
&=  - T ( (a \cdot u) \cdot T(v)) - T ( T(a \cdot u) \cdot  v) - TH ( T(a \cdot u), Tv) - T( H (a, Tu) \cdot  Tv) - T ( TH (a, Tu) \cdot v) \\
& - TH ( TH(a, Tu), Tv) - T ((aT(u)) \cdot v) +  T ( T(a \cdot u) \cdot v) + T ( TH (a, Tu) \cdot v) - TH ( a T(u), Tv)  + TH ( T(a \cdot u), Tv) \\
& + TH ( TH (a, Tu), Tv) + T( a \cdot (u \cdot T(v))) + T (a \cdot ((Tu) \cdot v)) + T ( a \cdot H (Tu, Tv)) + TH (a, T(u) T(v))\\
& \stackrel{\text{after cancellations}}{=}  - T( H (a, Tu) \cdot  Tv) - TH ( a T(u), Tv) + T ( a \cdot H (Tu, Tv)) + TH (a, T(u) T(v)) \\
&= T((\delta_{\mathrm{Hoch}} H)(a, Tu, Tv)) = 0.
\end{align*}
This shows that $l_T, r_T$ defines an $(M, *)$-bimodule structure on $A$.
\end{proof}

The previous proposition says that one can construct the Hochschild cohomology of the associative algebra $(M, *)$ with coefficients in the $M$-bimodule $A$. More precisely, the corresponding Hochschild cohomology is given by the cohomology of the cochain complex $\{ C^\bullet_{\mathrm{Hoch}} (M, A), \delta_{\mathrm{Hoch}} \}$, where $C^n_{\mathrm{Hoch}} (M, A) := \mathrm{Hom}(M^{\otimes n}, A)$, for $n \geq 0$ and the coboundary map $\delta_{\mathrm{Hoch}} : C^{n}_{\mathrm{Hoch}} (M, A) \rightarrow C^{n+1}_{\mathrm{Hoch}} (M, A)$ given by
\begin{align}\label{hoch-new-diff}
&(\delta_{\mathrm{Hoch}} (f) )(u_1, \ldots, u_{n+1}) \\
&= T(u_1) f (u_2, \ldots, u_{n+1}) - T (u_1 \cdot f (u_2, \ldots, u_{n+1})) - TH ( Tu_1, f (u_2, \ldots, u_{n+1})) \nonumber\\
&+ \sum_{i=1}^n (-1)^i ~ f (u_1, \ldots, u_{i-1}, u_i \cdot T(u_{i+1}) + T(u_i) \cdot  u_{i+1} + H (Tu_i, Tu_{i+1}), u_{i+1}, \ldots, u_{n+1}) \nonumber\\
&+ (-1)^{n+1} f(u_1, \ldots, u_n) T(u_{n+1}) - (-1)^{n+1} T ( f(u_1, \ldots, u_n) \cdot u_{n+1}) - (-1)^{n+1} TH ( f(u_1, \ldots, u_n) , Tu_{n+1} ), \nonumber
\end{align}
for $f \in C^n_{\mathrm{Hoch}}(M, A)$.
We denote the corresponding Hochschild cohomology groups by $H^\bullet_{\mathrm{Hoch}}(M, A)$. Thus, we have
\begin{align*}
H^0_{\mathrm{Hoch}} (M, A) = \{ a \in A |~ a T(u) - T(u) a + TH (Tu, a) - TH (a, Tu) = T(a \cdot u -u \cdot a),~ \forall u \in M \}.
\end{align*}


It follows from (\ref{hoch-new-diff}) that a linear map $f : M \rightarrow A$ is a Hochschild $1$-cocycle if and only if $f$ satisfies
\begin{align*}
T(u) f (v) + f (u) T(v) =~& T ( u \cdot f(v) + H (Tu, f(v))) + T ( f(u) \cdot v + H ( f(u), Tv )) \\
&+ f (u \cdot T(v) + T(u) \cdot v + H (Tu, Tv)),~ \text{ for all } u, v \in M. 
\end{align*}
This cocycle condition is same as the cocycle condition (\ref{dt-1-co}) in the cohomology complex of $T$. In fact, we have the following more general result.

\begin{prop}
Let $T : M \rightarrow A$ be an $H$-twisted Rota-Baxter operator. Then for any $f \in \mathrm{Hom}(M^{\otimes n}, A)$, we have
\begin{align*}
d_T (f) = (-1)^n~ \delta_{\mathrm{Hoch}} (f).
\end{align*}
\end{prop}

\begin{proof}
The explicit description of the bracket (\ref{derived-brkt}) yields  (see also \cite{das-rota})
\begin{align*}
\llbracket T, f \rrbracket =~& (-1)^n \big\{   T(u_1) f (u_2, \ldots, u_{n+1}) - T (u_1 \cdot f (u_2, \ldots, u_{n+1}))  \nonumber\\
&+ \sum_{i=1}^n (-1)^i ~ f (u_1, \ldots, u_{i-1}, u_i \cdot T(u_{i+1}) + T(u_i) \cdot  u_{i+1} + H (Tu_i, Tu_{i+1}), u_{i+1}, \ldots, u_{n+1}) \nonumber\\
&+ (-1)^{n+1} f(u_1, \ldots, u_n) T(u_{n+1}) - (-1)^{n+1} T ( f(u_1, \ldots, u_n) \cdot u_{n+1}) \big\}.
\end{align*}
Moreover, it is straightforward to observe from (\ref{ter-brk}) that
\begin{align*}
\llbracket T, T, f \rrbracket =& - 2 (-1)^n  \big\{ - TH ( Tu_1, f (u_2, \ldots, u_{n+1}))  - (-1)^{n+1} TH ( f(u_1, \ldots, u_n) , Tu_{n+1} ) \\
&+ \sum_{i=1}^n (-1)^i f (u_1, \ldots, u_{i-1}, H (Tu_i, Tu_{i+1}), u_{i+1}, \ldots, u_{n+1})  \big\}.
\end{align*} 
Hence we get
\begin{align*}
d_T (f ) = \llbracket T, f \rrbracket - \frac{1}{2} \llbracket T, T, f \rrbracket = (-1)^n~ \delta_{\mathrm{Hoch}} (f).
\end{align*}
\end{proof}

As a consequence of the previous proposition, we get the following.

\begin{thm}
Let $T : M \rightarrow A$ be an $H$-twisted Rota-Baxter operator. Then the cohomology $\mathcal{H}^\bullet_T (M, A)$ is isomorphic to the Hochschild cohomology $H^\bullet_{\mathrm{Hoch}}(M, A)$ of the associative algebra $(M, *)$ with coefficients in the $M$-bimodule $A.$
\end{thm}

\section{Deformations}\label{sec-def}
In this section, we study linear deformations and formal deformations of a twisted Rota-Baxter operator. We introduce Nijenhuis elements associated with a twisted Rota-Baxter operator that arise from trivial linear deformations. We also find a sufficient condition (in terms of Nijenhuis elements) for the rigidity of a twisted Rota-Baxter operator. 

\subsection{Linear deformations}
Let $A$ be an associative algebra, $M$ be an $A$-bimodule and $H$ be a Hochschild $2$-cocycle. Suppose $T: M \rightarrow A$ is an $H$-twisted Rota-Baxter operator. A linear deformation of $T$ is given by a parametrized sum $T_t = T + t T_1$, for some linear map $T_1 \in \mathrm{Hom}(M,A)$ so that $T_t$ is an $H$-twisted Rota-Baxter operator for all values of $t$. In this case, we say that $T_1$ generates a linear deformation of $T$.

Therefore, in a linear deformation $T_t = T + t T_1$, we must have
\begin{align*}
T_t (u) T_t (v) = T_t ( u \cdot T_t(v) + T_t (u) \cdot v + H (T_t (u), T_t (v))),
\end{align*}
for $u, v \in M$ and for all $t$. By equating coefficients of various powers of $t$, we get
\begin{align}
T(u) T_1(v) + T_1(u) T(v) =~& T_1 ( u \cdot T(v) + T(u) \cdot v + H (Tu, Tv)) \label{first-lin}\\
&+ T \big(  u \cdot T_1(v) + T_1(u) \cdot v + H (T(u), T_1(v)) + H ( T_1(u), T(v)) \big), \nonumber \\
T_1(u) T_1(v) =~& T  ( H (T_1(u), T_1(v)) ) + T_1 \big( u \cdot T_1(v) + T_1(u) \cdot v + H (T(u), T_1(v)) + H (T_1(u), T(v)) \big), \\
0 =~& T_1 ( H (T_1(u), T_1(v) ) ).
\end{align}
Observe that the condition (\ref{first-lin}) is equivalent that $T_1$ is a $1$-cocycle in the cohomology of $T$.

\begin{defn}
Two linear deformations $T_t = T + t T_1$ and $T_t' = T + t T_1'$ of an $H$-twisted Rota-Baxter operator $T$ are said to be equivalent if there exists $a \in A$ such that
\begin{align*}
\big( \phi_t = \mathrm{id}_A + t (\mathrm{ad}^l_a - \mathrm{ad}^r_a  ),~ \psi_t = \mathrm{id}_M + t (l_a-r_a + H (a, T -) - H (T-, a)) \big)
\end{align*}
is a morphism of $H$-twisted Rota-Baxter operators from $T_t$ to $T_t'$.
\end{defn}

Thus, the map $\phi_t : A  \rightarrow A$ is an algebra morphism which implies
\begin{align}\label{ref-t}
(ab -ba) (ac-ca) = 0, ~ \text{ for } b, c \in A.
\end{align}
The condition (\ref{action-comp}) is equivalent to the followings
\begin{align}\label{ref-u}
\begin{cases}
a \cdot ( b \cdot u) - ( b \cdot u ) \cdot a + H (a, T (b \cdot u)) - H ( T (b \cdot u), a)  \\ \qquad \qquad = (ab - ba) \cdot u + b \cdot ( a \cdot u - u \cdot a + H (a, Tu) - H (Tu, a)),\\
 (ab-ba) \cdot ( a \cdot u - u \cdot a + H (a, Tu) - H (Tu, a)) = 0, ~ \text{ for } b \in A, u \in M.
\end{cases}
\end{align}
\begin{align}\label{ref-v}
\begin{cases}
a \cdot ( u \cdot b ) - (u \cdot b ) \cdot a + H (a, T (u \cdot b)) - H ( T (u \cdot b), a)\\
\qquad \qquad = u \cdot ( ab -ba) + ( a \cdot u - u \cdot a + H (a, Tu) - H (Tu, a)) \cdot b, \\
(a \cdot u -u \cdot a + H (a, Tu) - H (Tu, a)) \cdot (ab  - ba) = 0, ~ \text{ for } b \in A, u \in M.
\end{cases}
\end{align}
Similarly, the conditions (\ref{h-comp}) and (\ref{t-comp}) are respectively equivalent to the followings
\begin{align}\label{ref-w}
\begin{cases}
a \cdot H (b,c) - H (b,c) \cdot a + H (a, TH (b,c)) - H (TH (b,c), a) = H (ab -ba, c) + H (b, ac-ca),\\
H (ab - ba, ac - ca) = 0, ~ \text{ for } b, c \in A.
\end{cases}
\end{align}
\begin{align}\label{t-comp-imply}
\begin{cases}
 T_1(u) + a T(u) - T(u) a  = T(a \cdot u -u \cdot a + H (a, Tu) - H (Tu, a)) + T_1'(u),\\
a T_1(u) - T_1(u) a = T_1' (a \cdot u-u \cdot a + H (a, Tu) - H (Tu, a)).
\end{cases}
\end{align}
Note that the first identity in (\ref{t-comp-imply}) implies that $T_1(u) - T_1'(u ) = d_T (a) (u)$, for $u \in M$. Hence we get the following.

\begin{thm}
Let $T_t = T + t T_1$ be a linear deformation of $T$. Then $T_1$ is a $1$-cocycle in the cohomology of $T$ whose class depends only on the equivalence class of the deformation $T_t$.
\end{thm}

Motivated from the above discussions, we introduce the following definition.

\begin{defn}
An element $a \in A$ is called a Nijenhuis element associated to the $H$-twisted Rota-Baxter operator $T$ if $a$ satisfies
\begin{align*}
a (l_T (u,a) - r_T (a, u)) - (l_T (u, a) - r_T (a, u)) a = 0
\end{align*}
and (\ref{ref-t}), (\ref{ref-u}), (\ref{ref-v}), (\ref{ref-w}) holds.
\end{defn}

The set of all Nijenhuis elements associated with $T$ is denoted by $\mathrm{Nij}(T)$. It follows from the previous discussions that a trivial linear deformation of $T$ gives rise to a Nijenhuis element. In the next subsection, we will consider rigidity of an $H$-twisted Rota-Baxter operator and find a necessary condition of the rigidity in terms of Nijenhuis elements.




\subsection{Formal deformations}
The classical formal deformation theory of Gerstenhaber \cite{gers} has been extended to Rota-Baxter operators in \cite{tang,das-rota}. In this subsection, we generalize it to $H$-twisted Rota-Baxter operators.

Let $A$ be an associative algebra, $M$ be an $A$-bimodule, and $H$ be a Hochschild $2$-cocycle. Note that the associative multiplication on $A$ induces an associative multiplication on the space $A[[t]]$ of formal power series in $t$ with coefficients from $A$. Moreover, the $A$-bimodule structure on $M$ induces an $A[[t]]$-bimodule structure on $M[[t]]$. The Hochschild $2$-cocycle $H$ also induces a Hochschild $2$-cocycle (denoted by the same notation) on $A[[t]]$ with coefficients in the bimodule $M[[t]]$.

\begin{defn}
Let $T: M \rightarrow A$ be a $H$-twisted Rota-Baxter operator. A formal one-parameter deformation of $T$ consists of a formal sum $T_t = \sum_{i \geq 0} t^i T_i \in \mathrm{Hom}(M,A)[[t]]$ with $T_0 = T$ such that $T_t : M [[t]] \rightarrow A[[t]]$ is an $H$-twisted Rota-Baxter operator. In other words,
\begin{align*}
T_t (u) T_t(v) = T_t ( u \cdot T_t(v) + T_t (u) \cdot v + H (T_t (u), T_t (v)) ), ~ \text{ for } u, v \in M.
\end{align*}
\end{defn}

Therefore, in a formal deformation, the following system of equations hold: for $n \geq 0$,
\begin{align*}
\sum_{i+j = n} T_i(u) T_j (v) = \sum_{i+j = n} T_i (u \cdot T_j(v) + T_j(u) \cdot v ) + \sum_{i+j+k = n} T_i ( H (T_j (u), T_k (v)) ).
\end{align*}
These equations are called deformation equations. The above equation holds for $n=0$ as $T_0 = T$ is an $H$-twisted Rota-Baxter operator. Similar to linear deformations, we get for $n=1$ that
$T_1$ is a $1$-cocycle in the cohomology of $T$, called the infinitesimal of the deformation.

\begin{defn}
Two formal one-parameter deformations $T_t = \sum_{i \geq 0} t^i T_i$ and $T_t' = \sum_{i \geq 0} t^i T_i'$ of an $H$-twisted Rota-Baxter operator $T$ are said to be equivalent if there exists an element $a \in A$, linear maps $\phi_i \in \mathrm{Hom}(A, A)$ and $\psi_i \in \mathrm{Hom}(M,M)$, for $i \geq 2$ such that
\begin{align*}
\big(   \phi_t = \mathrm{id}_A + t (\mathrm{ad}^l_a  - \mathrm{ad}^r_a ) + \sum_{i \geq 2} t^i \phi_i, ~ 
\psi_t = \mathrm{id}_A + t (l_a - r_a + H (a, T-) - H (T-, a)) + \sum_{i \geq 2} t^i \psi_i \big)
\end{align*}
is a morphism of $H$-twisted Rota-Baxter operators from $T_t$ to $T_t'$.
\end{defn}


Again similar to linear deformations, we get
\begin{align*}
T_1(u) - T_1' (u)= d_T (a) (u), ~ \text{ for } u \in M.
\end{align*}
Hence we obtain the following.
\begin{thm}\label{inf-1-coc}
Let $T_t = \sum_{i \geq 0} t^i T_i$ be a formal one-parameter deformation of an $H$-twisted Rota-Baxter operator $T$. Then the linear term $T_1$ is a $1$-cocycle in the cohomology of $T$ whose cohomology class depends only on the equivalence class of the deformation $T_t$.
\end{thm}

In the following theorem, we give a sufficient condition on an $H$-twisted Rota-Baxter operator $T$ which ensures that any deformation of $T$ is equivalent to the undeformed one.

\begin{thm}\label{rigid-suff}
Let $T$ be an $H$-twisted Rota-Baxter operator satisfying $\mathcal{Z}^1_T (M, A) = d_T (\mathrm{Nij}(T)).$ Then any formal deformation $T_t$ of $T$ is equivalent to the undeformed one $T_t' = T$.
\end{thm}

\begin{proof}
Let $T_t = \sum_{i \geq 0} t^i T_i$ be any formal deformation of $T$. Then by the previous theorem, the linear term $T_1$ is in $\mathcal{Z}^1_T (M,A)$. Hence by the assumption, there exists $a \in \mathrm{Nij}(T)$ such that $T_1 = d_T (a)$. We set
\begin{align*}
\phi_t := \mathrm{id}_A + t (\mathrm{ad}^l_a - \mathrm{ad}^r_a )  ~~~ \text{ and } ~~~ \psi_t = \mathrm{id}_M + t (l_a - r_a + H (a, T -) - H (T-, a))
\end{align*}
and define $T_t' := \phi_t \circ T_t \circ \psi_t^{-1}$. Then by definition, $T_t$ is equivalent to $T_t'$. Moreover, we have
\begin{align*}
T_t'(u) ~~~~ (\mathrm{mod }~~t^2) =~&  \big( \mathrm{id}_A + t (\mathrm{ad}^l_a - \mathrm{ad}^r_a)  \big) \circ (T + tT_1) \circ (\mathrm{id}_M - t (l_a - r_a + H (a, T -) - H (T-, a)))(u)  ~~~~ (\mathrm{mod }~~t^2) \\
=~& \big( \mathrm{id}_A + t (\mathrm{ad}^l_a  - \mathrm{ad}^r_a )  \big)  (Tu + tT_1 (u) - tT (a \cdot u-u \cdot a + H (a, Tu) - H (Tu, a) ) )  ~~~~ (\mathrm{mod }~~t^2) \\
=~& Tu + t (  T_1(u) - T(a \cdot u-u \cdot a) + a T(u) - T(u) a - TH (a, Tu) + TH (Tu, a) ) \\
=~& Tu ~~~~\qquad (\mathrm{as }~~~ T_1(u) = d_T (a) (u)).
\end{align*}
This shows that the coefficient of $t$ in the deformation $T_t'$ vanishes. By applying the same process, one obtains the equivalence between the deformations $T_t$ and $T$. Hence the proof.
\end{proof}

\begin{remark}
An $H$-twisted Rota-Baxter operator $T$ is said to be rigid if any formal deformation of $T$ is equivalent to the undeformed one. Thus, it follows from the previous theorem that the condition $\mathcal{Z}^1_T (M,A) = d_T (\mathrm{Nij}(T))$ is sufficient for the rigidity of the operator $T$.
\end{remark}

\subsection{Applications to Reynolds operators}

It has been observed in section \ref{sec-2} that Reynolds operators on an associative algebra are $H$-twisted Rota-Baxter operators wit $H = - \mu$, where $\mu : A^{\otimes 2} \rightarrow A$ is the associative multiplication on $A$. This implies that the results of previous sections can be applied to Reynolds operators. Here we only highlight some of the corresponding results.

The following result is the Reynolds operator version of Theorem \ref{mc-main}.

\begin{thm}
Let $A$ be an associative algebra. Then the graded vector space $\bigoplus_{n \geq 0} \mathrm{Hom} (A^{\otimes n}, A)$ carries an $L_\infty$-algebra with structure maps
\begin{align*}
l_2 (P, Q) =~& \llbracket P, Q \rrbracket, \\
l_3 (P, Q, R) =~& \llbracket P, Q, R \rrbracket.
\end{align*}
and higher trivial maps. (Note that we need to substitute $H = - \mu$ in the definition of the ternary bracket.) Moreover, there is a one-to-one correspondence between Maurer-Cartan elements in the above $L_\infty$-algebras and Reynolds operators on $A$.  
\end{thm}

Let $R$ be a Reynolds operator on $A$. The cohomology induced from the corresponding Maurer-Cartan element is called the cohomology of $R$.

A Reynolds operator $R$ induces a new associative algebra structure on $A$ given by
\begin{align*}
a * b = R(a) b + a R(b) - R(a) R(b), ~ \text{ for } a, b \in A.
\end{align*}
Moreover, the vector space $A$ is an $(A, *)$-bimodule given by
\begin{align*}
l (a, b ) = R(a) b  + R ( R(a) b) - R ( ab ), \qquad r (b, a) = b R(a) + R (b R(a)) - R ( ba), ~ \text{ for } a \in (A, *), b \in A.
\end{align*}
Note that this is not the adjoint bimodule.
The Hochschild cohomology groups of the algebra $(A, *)$ with coefficients in the above bimodule are isomorphic to the cohomology of $R$.

Deformations of a Reynolds operator $R$ can be defined along the line of section \ref{sec-def}. Such deformations are governed by the cohomology of $R$ as of Theorem \ref{inf-1-coc} and Theorem \ref{rigid-suff}.

\section{NS-algebras}\label{sec-ns-alg}
In this section, we define cohomology of NS-algebras using non-symmetric multiplicative operads. As a consequence, we get a Gerstenhaber algebra structure on the cohomology. 
Finally, we study deformations of NS-algebras from cohomological points of view.

\subsection{A new non-symmetric operad and cohomology of NS-algebras}
Here we construct a new non-symmetric operad associated to a vector space $A$. A multiplication on this operad is equivalent to an NS-algebra structure on $A$. The cohomology induced from the multiplication is the cohomology of the given NS-algebra.

Let $C_1 = \{ 1 \}$ and $C_n$ be the set of first $n+1$ natural numbers, for $n \geq 2$. For convenience, we write $C_1 = \{ [1] \}$ and $C_n = \{ [1], \ldots, [n], [n+1] \}$, for $n \geq 2$. Given a vector space $A$, we consider the collection of vector spaces
\begin{align*}
\mathcal{O}_A(n) = \mathrm{Hom}( \mathbb{K}[C_n] \otimes A^{\otimes n}, A), ~\text{ for } n \geq 1.
\end{align*}
For any $f \in \mathcal{O}_A(m), g \in \mathcal{O}_A(n)$ and $1 \leq i \leq m$, we define $f \circ_i g \in \mathcal{O}_A(m+n-1)$ by the following\\
\begin{align}
( f \circ_i g) ([r]; a_1, \ldots, a_{m+n-1}) = \label{new-op-pc}
\end{align}
\begin{align}
\begin{cases} 
\medskip \medskip
f([r]; a_1, \ldots, a_{i-1}, g ([1]+\cdots+[n+1]; a_i, \ldots, a_{i+n-1}), a_{i+n}, \ldots, a_{m+n-1}) & \text{ if } 1 \leq r \leq i-1, \nonumber \\ \medskip \medskip
f ([i]; a_1, \ldots, a_{i-1}, g ([r-i+1]; a_i, \ldots, a_{i+n-1}), a_{i+n}, \ldots, a_{m+n-1}) & \text{ if } i \leq r \leq i+n-1,   \nonumber \\  \medskip \medskip
f([r-n+1]; a_1, \ldots, a_{i-1}, g ([1]+\cdots+[n+1]; a_i, \ldots, a_{i+n-1}), a_{i+n}, \ldots, a_{m+n-1}) & \text{ if } i+n \leq r \leq m+n-1,  \nonumber  \\ 
f ([i]; a_1, \ldots, a_{i-1}, g ([n+1]; a_i, \ldots, a_{i+n-1}), a_{i+n}, \ldots, a_{m+n-1}) & \text{ if }  r = m+n.  \nonumber \\
+ f ([m+1]; a_1, \ldots, a_{i-1}, g ([1]+ \cdots+[n+1]; a_i, \ldots, a_{i+n-1}), a_{i+n}, \ldots, a_{m+n-1}).  \nonumber \\
\end{cases}
\end{align}
Here $a_1, \ldots, a_{m+n-1} \in A.$
\medskip

\medskip

With the above notations, we have the following.

\begin{thm}\label{thm-new-operad}
For any vector space $A$, the collection of spaces $\{ \mathcal{O}_A(n)\}_{n \geq 1}$ with partial compositions 
\begin{align*}
\circ_i : \mathcal{O}_A(m ) \otimes \mathcal{O}_A(n) \rightarrow \mathcal{O}_A(m+n-1), ~\text{ for } m, n \geq 1 \text{ and } 1 \leq i \leq m
\end{align*}
is a non-symmetric operad with the identity element $\mathrm{id} \in \mathcal{O}_A(1)$ given by $\mathrm{id}([1]; a) = a$, for all $a \in A$.

Moreover, an NS-algebra structure on $A$ is equivalent to a multiplication $\pi \in \mathcal{O}_A(2)$ on the operad $\mathcal{O}_A$ (i.e., $\pi \in \mathcal{O}_A(2)$ that satisfies $\pi \circ_1 \pi = \pi \circ_2 \pi$).
\end{thm}

\begin{proof}
To prove that the collection of spaces $\{ \mathcal{O}_A(n) \}_{n \geq 1}$ with partial compositions $\circ_i$ forms a non-symmetric operad, we need to verify that
\begin{align}
( f \circ_i g) \circ_{i+j-1} h =~& f \circ_i ( g \circ_j h), ~ \text{ for } 1 \leq i \leq m,~ 1 \leq j \leq n,\\
(f \circ_i g) \circ_{j+n-1} h =~& ( f \circ_j h) \circ_i g, ~ \text{ for } 1 \leq i < j \leq m,
\end{align}
for $f \in \mathcal{O}_A(m),~ g \in \mathcal{O}_A(n),~ h \in \mathcal{O}_A(p)$, and the element $\mathrm{id}\in \mathcal{O}_A(1)$ satisfies $f \circ_i \mathrm{id} = f = f \circ_1 f$, for $f \in \mathcal{O}_A(m)$ and $1 \leq i \leq m$ \cite{gers-voro,lod-val-book}. Take $m, n , p \geq 2$ (if some of them are $1$, then the calculations are more simpler) and $1 \leq i \leq m$, $i \leq j \leq n$. For $1 \leq r \leq i-1$, we have
\begin{align*}
&((f \circ_i g) \circ_{i+j-1} h ) ([r]; a_1, \ldots, a_{m+n+p-2}) \\
&= (f \circ_i g) ( [r]; a_1, \ldots, h ([1]+ \cdots+ [p+1]; a_{i+j-1}, \ldots, a_{i+j+p-2}), a_{i+j+n-1}, \ldots , a_{m+n+p-2}) \\
&= f \big([r]; a_1, \ldots, g \big([1]+ \cdots + [n+1]; a_i, \ldots, h ([1]+\cdots +[p+1]; a_{i+j-1}, \ldots, a_{i+j+p-2}),  \ldots , a_{i+n+p-2}  \big),\\
& \qquad \qquad \qquad \qquad \qquad \qquad \qquad \qquad \qquad \qquad \qquad \qquad \qquad \qquad \qquad \qquad \qquad \ldots, a_{m+n+p-2}     \big).
\end{align*} 
On the other hand,
\begin{align}\label{otoh}
&(f \circ_i (g \circ_j h)) ([r]; a_1, \ldots, a_{m+n+p-2}) \nonumber \\
&= f ( [r]; a_1, \ldots, a_{i-1}, (g \circ_j h) ([1]+ \cdots+[n+p]; a_i, \ldots, a_{i+n+p-2}), \ldots, a_{m+n+p-2}).
\end{align}
Note that
\begin{align*}
&(g \circ_j h) ([1]+ \cdots+[n+p]; a_i, \ldots, a_{i+n+p-2}) \\
&= (g \circ_j h) ([1]+ \cdots+[j-1]; a_i, \ldots, a_{i+n+p-2}) + (g \circ_j h) ([j]+ \cdots+[j+p-1]; a_i, \ldots, a_{i+n+p-2}) \\
&+ (g \circ_j h) ([j+p]+ \cdots+[n+p-1]; a_i, \ldots, a_{i+n+p-2}) + (g \circ_j h) ([n+p]; a_i, \ldots, a_{i+n+p-2}) \\
&= g ([1]+ \cdots + [j-1]; a_i, \ldots, h ([1]+ \cdots + [p+1]; a_{i+j-1}, \ldots, a_{i+j+p-2}), \ldots, a_{i+n+p-2} ) \\
&+ g ([j]; a_i, \ldots, h ([1]+ \cdots + [p]; a_{i+j-1}, \ldots, a_{i+j+p-2} ), \ldots, a_{i+n+p-2} )\\
&+ g ([j+1]+ \cdots + [n]; a_i, \ldots, h ( [1]+\cdots+[p+1]; a_{i+j-1}, \ldots, a_{i+j+p-2}), \ldots, a_{i+n+p-2} ) \\
&+ g ([j] ; a_i, \ldots, h ([p+1]; a_{i+j-1}, \ldots, a_{i+j+p-2}), \ldots, a_{i+n+p-2} ) \\
&+ g ([n+1]; a_i, \ldots, h ([1]+\cdots+[p+1]; a_{i+j-1}, \ldots, a_{i+j+p-2}), \ldots, a_{i+n+p-2} ) \\
&= g ([1]+\cdots+[n+1]; a_i, \ldots, h ([1]+\cdots+[p+1]; a_{i+j-1}, \ldots, a_{i+j+p-2}), \ldots, a_{i+n+p-2} ). 
\end{align*}
By substituting this in (\ref{otoh}), we get that 
\begin{align*}
((f \circ_i g) \circ_{i+j-1} h )([r]; a_1, \ldots, a_{m+n+p-2})  =  (f \circ_i (g \circ_j h)) ([r]; a_1, \ldots, a_{m+n+p-2})
\end{align*}
holds, for $1 \leq r \leq i-1$. Similarly, it holds if $r$ belongs to either of the intervals $i \leq r \leq i+j-2$ / $i+j-1 \leq r \leq i+j+p-2$ / $i+j+p-1 \leq r \leq i+n+p-2$ / $i+n+p-1 \leq r \leq m+n+p-2$ / $r = m+n+p-1$. Hence we have $(f \circ_i g) \circ_{i+j-1} h  = f \circ_i (g \circ_j h)$.

In a similar way, one can show that $(f \circ_i g) \circ_{j+n-1} h = (f \circ_j h) \circ_i g$, for $1 \leq i < j \leq m.$ Finally, it is easy to observe that the element $\mathrm{id} \in \mathcal{O}_A(1)$ satisfies  $f \circ_i \mathrm{id} = f = \mathrm{id} \circ_1 f$, for $f \in \mathcal{O}_A(m)$ and $1 \leq i \leq m$. Therefore, the collection of spaces $\{ \mathcal{O}_A(n) \}_{n \geq 1}$ with the partial compositions $\circ_i$ forms a non-symmetric operad.\\

\noindent {\em Proof of the last part.}

Let $\pi \in \mathcal{O}_A(2)$ be a multiplication on the operad $\mathcal{O}_A$. Then $\pi$ is equivalent to three binary operations (say $\prec, \succ, \curlyvee$) on $A$ given by
\begin{align*}
a \prec b = \pi ([1]; a, b ), \quad a \succ b = \pi ([2]; a, b ) ~~~ \text{ and } ~~~ a \curlyvee b = \pi ([3]; a, b), ~ \text{ for } a, b \in A.
\end{align*} 

Note that
\begin{align*}
(\pi \circ_1 \pi - \pi \circ_2 \pi) ([r]; a, b, c) 
= \begin{cases} 
(a \prec b) \prec c - a \prec (b * c) & \text{ if } [r] = [1], \\
(a \succ b) \prec c - a \succ (b \prec c)  & \text{ if } [r] = [2], \\
( a * b ) \succ c - a \succ ( b \succ c)  & \text{ if } [r] = [3], \\
( a \curlyvee b) \prec c + ( a * b ) \curlyvee c - a \succ (b \curlyvee c) - a \curlyvee ( b * c) & \text{ if } [r] = [4]. \\
\end{cases}
\end{align*}
This shows that $\pi$ is a multiplication on the operad $\mathcal{O}_A$ (i.e., $\pi \circ_1 \pi = \pi \circ_2 \pi$) if and only if the corresponding binary operations $(\prec, \succ, \curlyvee)$ defines an NS-algebra structure on $A$.
\end{proof}

\begin{remark}
Consider a collection of subspaces $\mathcal{O}_A'(n ) \subset \mathcal{O}_A(n)$, for $n \geq 1$ as follows. Take $\mathcal{O}_A'(1) = \mathcal{O}_A(1)$ and for $n \geq 2$,
\begin{align*}
\mathcal{O}_A'(n) = \{ f \in \mathcal{O}_A(n) |~ f_{[n+1]} = 0 \}.
\end{align*}
Then the collection of spaces $\{\mathcal{O}_A'(n)\}_{n \geq 1}$ with the same partial compositions forms a suboperad of the operad $\mathcal{O}_A$. This suboperad is the operad in which multiplication is equivalent to a dendriform structure on $A$. This operad is explicitly treated in \cite{das-dend}.
\end{remark}

Let $(A, \prec, \succ, \curlyvee)$ be an NS-algebra. Then it follows from the previous theorem that the NS-algebra structure on $A$ is equivalent to a multiplication $\pi \in \mathcal{O}_A(2) = \mathrm{Hom}( \mathbb{K}[C_2] \otimes A^{\otimes 2}, A)$ in the operad $\mathcal{O}_A$. Hence $\pi$ induces a differential $\delta_\pi : \mathcal{O}_A(n) \rightarrow \mathcal{O}_A(n+1)$ as follows
\begin{align*}
\delta_\pi (f) := (-1)^{n-1} \big[ \sum_{i =1}^2 (-1)^{(i-1)(n-1)}~ \pi \circ_i f -(-1)^{n-1} \sum_{i=1}^n (-1)^{(i-1)}~ f \circ_i \pi \big].
\end{align*}
The cohomology of the cochain complex $\{ \mathrm{Hom}(\mathbb{K}[C_\bullet]\otimes A^{\otimes \bullet}, A), \delta_\pi \}$ is called the cohomology of the NS-algebra $A$ with coefficients in itself and denoted by $H^\bullet_{\mathrm{NS}}(A, A)$. Since the cohomology is induced from an operad with multiplication, the graded space of cohomology $H^\bullet_{\mathrm{NS}} (A, A)$ inherits a Gerstenhaber structure. See \cite{gers-voro,das-loday} for more details about the Gerstenhaber structure induced from a non-symmetric operad with multiplication.

\begin{remark}
Note that representations of NS-algebras and cohomology with coefficients in a representation can be easily defined without much hard work. See \cite{das-dend} for the case of dendriform algebras.
\end{remark}

\subsection{Relation with the Hochschild cohomology}

We have seen that the sum of three operations in an NS-algebra $A$ is associative. Therefore, we have two cohomologies, namely, the cohomology of the NS-algebra and the Hochschild cohomology of the induced associative algebra structure on $A$. We relate these two cohomologies using the following result.

\begin{thm}\label{operad-split}
Let $A$ be a vector space. Then the collection $\{ \Theta_n \}_{n \geq 1}$ of maps 
\begin{align*}
\Theta_n : \mathrm{Hom}(\mathbb{K}[C_n] \otimes A^{\otimes n}, A) \rightarrow \mathrm{Hom}(A^{\otimes n}, A), ~~\text{ for } n \geq 1
\end{align*}
defined by
\begin{align*}
\Theta_n (f) = \begin{cases} f_{[1]} & \text{ if } n =1 \\
f_{[1]} + \cdots + f_{[n]} + f_{[n+1]} & \text{ if } n \geq 2 \end{cases}
\end{align*}
is a morphism of non-symmetric operads. Here $\{ \mathrm{Hom}(\mathbb{K}[C_n] \otimes A^{\otimes n}, A) \}_{n \geq 1}$ is equipped with the operad structure given in Theorem \ref{thm-new-operad} and $\{ \mathrm{Hom}(A^{\otimes n}, A) \}_{n \geq 1}$ is equipped with the standard endomorphism operad structure.
\end{thm}

\begin{proof}
To prove that the collection $\{ \Theta_n \}_{n \geq 1}$ of maps is a morphism of non-symmetric operads, we need to show that 
\begin{align*}
\Theta_{m+n-1} ( f \circ_i g) = \Theta_m (f) \circ_i \Theta_n (g) \qquad \text{ and } \qquad \Theta_1 (\mathrm{id}) = \mathrm{id}_A,
\end{align*}
for $f \in \mathrm{Hom}( \mathbb{K}[C_m] \otimes A^{\otimes m }, A),~ g \in \mathrm{Hom}( \mathbb{K}[C_n] \otimes A^{\otimes n }, A)$ and $1 \leq i \leq m$. Here on the right hand side, we use partial compositions of endomorphism operad $\mathrm{End}_A$ (however we denote by the same notation as $\circ_i$). First we take $m,n \geq 2$. However the case when one (or both) of $m,n$ is $1$ involves more simple calculation. Observe that
\begin{align*}
&\Theta_{m+n-1} ( f \circ_i g) \\
&= ( f \circ_i g)_{[1]} + \cdots + ( f \circ_i g)_{[m+n-1]} + ( f \circ_i g)_{[m+n]} \\
&= \sum_{r=1}^{i-1} ( f \circ_i g)_{[r]} + \sum_{r=i}^{i+n-1} (f \circ_i g)_{[r]} + \sum_{r=i+n}^{m+n-1} (f \circ_i g)_{[r]} + (f \circ_i g)_{[m+n]} \\
&= (f_{[1]} + \cdots + f_{[i-1]}) \circ_i (g_{[1]} + \cdots + g_{[n+1]})  + f_{[i]} \circ_i (g_{[1]} + \cdots + g_{[n]}) \\
~&+ (f_{[i+1]} + \cdots + f_{[m]}) \circ_i (g_{[1]} + \cdots + g_{[n+1]})  + f_{[i]} \circ_i g_{[n+1]} + f_{[m+1]} \circ_i(g_{[1]} + \cdots + g_{[n+1]}) \quad  (\text{by } (\ref{new-op-pc}))\\
&= (f_{[1]} + \cdots + f_{[m+1]} )  \circ_i  (g_{[1]} + \cdots + g_{[n+1]})  = \Theta_m (f) \circ_i \Theta_n (g).
\end{align*}
Finally, it follows from the definition of $\mathrm{id} \in \mathrm{Hom}( \mathbb{K}[C_1] \otimes A,A)$ that $\Theta_1 (\mathrm{id} ) = \mathrm{id}_A$. Hence the proof.
\end{proof}

Let $(A, \prec, \succ, \curlyvee)$ be an NS-algebra with the corresponding associative product $* = \prec + \succ + \curlyvee$. If $\pi \in \mathrm{Hom}(\mathbb{K}[C_2] \otimes A^{\otimes 2}, A)$ denotes the multiplication on the operad $\mathcal{O}_A$ defining the NS-algebra structure on $A$, then we have
\begin{align*}
\Theta_2 (\pi) = *.
\end{align*}
This shows that the multiplication in the operad $\mathcal{O}_A$ for the NS-algebra structure on $A$ corresponds the (associative) multiplication in the endomorphism operad via $\Theta_2$. Hence by the consequence of Theorem \ref{operad-split}, the collection of maps $\{ \Theta_n \}_{n \geq 1}$ induces a morphism $\Theta_\bullet : H^\bullet_{\mathrm{NS}}(A, A) \rightarrow H^\bullet_{\mathrm{Hoch}}(A, A)$ between corresponding cohomologies. This is in fact a morphism of Gerstenhaber algebras.


\subsection{Deformations of NS-algebras}
In this subsection, we study deformations of NS-algebras. Since our approaches and results are similar to the classical case of Gerstenhaber \cite{gers}, we only state some of the results without proof.

Let $A = (A, \prec, \succ, \curlyvee)$ be an NS-algebra. Consider the space $A[[t]]$ of formal power series in $t$ with coefficients from $A$. Then $A[[t]]$ is a $\mathbb{K}[[t]]$-module.
\begin{defn}
A formal one-parameter deformation of $A$ consists of formal sums 
\begin{align*}
\prec_t = \sum_{i \geq 0} t^i \prec_i, \quad \succ_t = \sum_{i \geq 0} t^i \succ_i \text{ and } \curlyvee_t = \sum_{ i \geq 0} t^i \curlyvee_i (\text{with }\prec_0 = \prec,~\succ_0 = \succ \text{ and } \curlyvee_0 = \curlyvee)
\end{align*}
of bilinear maps on $A$ such that $(A[[t]] , \prec_t , \succ_t, \curlyvee_t)$ is an NS-algebra over the base $\mathbb{K}[[t]]$.
\end{defn}

Let $T : M \rightarrow A$ be an $H$-twisted Rota-Baxter operator. Suppose $T_t$ is a formal deformation of $T$. Then $T_t$ induces a formal deformation of the induced NS-algebra structure on $M$ given by
\begin{align*}
u \prec_t v =  \sum_{i \geq 0} u \cdot T_i (v)  \qquad   u \succ_t v =  \sum_{i \geq 0} T_i (u) \cdot  v   \qquad u \curlyvee_t v =  \sum_{ i \geq 0} t^i (\sum_{j+k =i} H(T_j(u), T_k(v))).
\end{align*}

Note that, in a formal deformation $(\prec_t, \succ_t, \curlyvee_t)$ of an NS-algebra, the following system of equations must hold: for $n \geq 0$,
\begin{align}
\sum_{i+j = n} (a \prec_i b) \prec_j c =~& \sum_{i+j = n} a \prec_i (b *_j c), \label{ns-def-1}\\
\sum_{i+j = n} (a \succ_i b ) \prec_j c =~& \sum_{i+j = n} a \succ_i (b \prec_j c), \label{ns-def-2}\\
\sum_{i+j = n} (a *_i b ) \succ_j c =~& \sum_{i+j = n} a \succ_i ( b \succ_j c), \label{ns-def-3}\\
\sum_{i+j = n} (a \curlyvee_i b ) \prec_j c + (a *_i b ) \curlyvee_j c =~& \sum_{i+j = n} a \succ_i (b \curlyvee_j c) + a \curlyvee_i (b *_j c), \label{ns-def-4}
\end{align}
for $a, b, c \in A$. For simplicity, we define an element $\pi_i \in \mathrm{Hom}(\mathrm{K}[C_2] \otimes A^{\otimes 2}, A)$, for each $i \geq 0$, by
\begin{align}\label{pi-i}
\pi_i ([1]; a, b) = a \prec_i b, \qquad \pi_i ([2]; a, b) = a \succ_i b ~~~ \text{ and } ~~~ \pi_i ([3]; a, b ) = a \curlyvee_i b.
\end{align}
Then we have $\pi_0 = \pi$. With the above notations, the identities (\ref{ns-def-1}) - (\ref{ns-def-4}) can be simply written as
\begin{align}
\sum_{i+j = n } \pi_i \circ_1 \pi_j = \sum_{i+j = n} \pi_i \circ_2 \pi_j.
\end{align} 
For $n = 1$, it follows that $\pi \circ_1 \pi_1 - \pi \circ_2 \pi_1  + \pi_1 \circ_1 \pi - \pi_1 \circ_2 \pi = 0$, or equivalently, $\delta_\pi (\pi_1) = 0$. This shows that $\pi_1 = (\prec_1, \succ_1, \curlyvee_1)$ is a $2$-cocycle in the cohomology of the NS-algebra $A$. This $2$-cocycle is called the infinitesimal of the deformation $(\prec_t, \succ_t, \curlyvee_t)$.

\begin{defn}
Two deformations  $(\prec_t, \succ_t, \curlyvee_t)$ and  $(\prec_t', \succ_t', \curlyvee_t')$ of an NS-algebra $A$ are said to be equivalent if there exists a formal isomorphism $\Phi_t = \sum_{i \geq 0} t^i \Phi_i : A[[t]] \rightarrow A[[t]]$ with $\Phi_0 = \mathrm{id}_A$ such that
\begin{align*}
\Phi_t : (A[[t]], \prec_t, \succ_t, \curlyvee_t) \rightarrow (A[[t]], \prec_t', \succ_t', \curlyvee_t')
\end{align*}
is a morphism of NS-algebras.
\end{defn}

The condition in the above definition is equivalent to
\begin{align*}
\sum_{i+j = n} \Phi_i \circ \pi_j = \sum_{i+j+k = n} \pi_i' \circ (\Phi_j \otimes \Phi_k),~ \text{ for } n \geq 0.
\end{align*}
For $n = 1$, we get
\begin{align*}
\pi_1 - \pi_1' = \pi \circ ( \Phi_1 \otimes \mathrm{id}) + \pi ( \mathrm{id} \otimes \Phi_1) - \Phi_1 \circ \pi = \delta_\pi (\Phi_1).
\end{align*}
Thus, we obtain the following.

\begin{thm}
Let $(\prec_t, \succ_t , \curlyvee_t)$ be a formal deformation of an NS-algebra $A$. Then the linear term $\pi_1 = (\prec_1, \succ_1, \curlyvee_1)$ is a $2$-cocycle in the cohomology of $A$ whose cohomology class depends only on the equivalence class of the deformation $(\prec_t, \succ_t, \curlyvee_t)$.
\end{thm}

\begin{defn}
An NS-algebra $A$ is said to be rigid if any deformation $(\prec_t, \succ_t, \curlyvee_t)$ is equivalent to the undeformed one $(\prec_t' = \prec, \succ_t' = \succ, \curlyvee_t' = \curlyvee).$
\end{defn}

In similar to the classical case of Gerstenhaber \cite{gers} (see also \cite{das-loday} in terms of non-symmetric operad with multiplication), we get the following.

\begin{thm}
If $H^2_{\mathrm{NS}}(A, A) = 0$ then the NS-algebra $A$ is rigid.
\end{thm}

One may also study finite order deformations of an NS-algebra and their higher order extensions. Let $A$ be an NS-algebra. An order $N$ deformation of $A$ consists of finite sums
\begin{align*}
\prec_t = \sum_{i = 0}^N t^i \prec_i \qquad \succ_t = \sum_{i = 0} t^i \succ_i ~~~ \text{ and } \curlyvee_t =  \sum_{i=0}^N t^i \curlyvee_i
\end{align*}
such that the space $A[[t]]/(t^{N+1})$ with bilinear operations $(\prec_t, \succ_t, \curlyvee_t)$ is an NS-algebra over the base $\mathbb{K}[[t]]/(t^{N+1}).$

\medskip

Thus, with the notations of (\ref{pi-i}), the following equations are hold
\begin{align}\label{fini-ord-eq}
\sum_{i+j = n} \pi_i \circ_1 \pi_j = \sum_{i+j=n} \pi_i \circ_2 \pi_j, ~ \text{ for } n=0, 1, \ldots, N.
\end{align}

A deformation $(\prec_t, \succ_t, \curlyvee_t)$ of order $N$ is said to be extensible if there exists a triple $(\prec_{N+1}, \succ_{N+1}, \curlyvee_{N+1})$ of bilinear operations on $A$ such that
$( \overline{\prec}_t = \prec_t + t^{N+1} \prec_{N+1}, \overline{\succ}_t = \succ_t + t^{N+1} \succ_{N+1}, \overline{\curlyvee}_t = \curlyvee_t + t^{N+1} \curlyvee_{N+1})$ is a deformation of order $N+1$.

In such a case, we define an element $\pi_{N+1} \in \mathrm{Hom}(\mathbb{K}[C_2] \otimes A^{\otimes 2}, A)$ by $$\pi_{N+1} ([1]; a, b) = a \prec_{N+1} b,~~\pi_{N+1} ([2]; a, b) = a \succ_{N+1} b ~~~\text{ and } ~~~ \pi_{N+1} ([3]; a, b) = a \curlyvee_{N+1} b.$$ Thus, in an extensible deformation, together with equations (\ref{fini-ord-eq}), one more equation needs to be satisfied, namely,
\begin{align*}
\sum_{i+j =N+1} \pi_i \circ_1 \pi_j = \sum_{i+j =N+1} \pi_i \circ_2 \pi_j, \quad \text{equavalently},
\end{align*}
\begin{align}\label{ord-n-1}
\pi \circ_1 \pi_{N+1} - \pi \circ_2 \pi_{N+1} +  \pi_{N+1} \circ_1 \pi - \pi_{N+1} \circ_2 \pi = - \sum_{i+j =N+1, i, j \geq 1} \big( \pi_i \circ_1 \pi_j - \pi_i \circ_2 \pi_j \big).
\end{align}
Note that the right-hand side of (\ref{ord-n-1}) is a $3$-cochain in the cohomology complex of the NS-algebra $A$ that depends only on the deformation $(\prec_t, \succ_t, \curlyvee_t)$. We denote this by $\mathrm{Ob}_{(\prec_t, \succ_t, \curlyvee_t)}$. Then similar to the classical theories, one can show that $\mathrm{Ob}_{(\prec_t, \succ_t, \curlyvee_t)}$ is a $3$-cocycle in the cohomology complex of $A$, hence, defines a cohomology class $[\mathrm{Ob}_{(\prec_t, \succ_t, \curlyvee_t)}] \in H^3_{\mathrm{NS}} (A, A).$
Thus from (\ref{ord-n-1}), we get the following.

\begin{thm}
A deformation $(\prec_t, \succ_t, \curlyvee_t)$ of order $N$ is extensible if and only if the obstruction class $[\mathrm{Ob}_{(\prec_t, \succ_t, \curlyvee_t)}] \in H^3_{\mathrm{NS}} (A, A)$ is trivial.
\end{thm}

\begin{corollary}
If $H^3_{\mathrm{NS}} (A, A) = 0$ then any finite order deformation of $A$ extends to deformation of next order.
\end{corollary}

\begin{remark}
In this section, we mainly concentrate on cohomology and deformations of NS-algebras. In \cite{das-rota} the author constructs a morphism from the cohomology of a Rota-Baxter operator to the cohomology of the corresponding dendriform algebra. In could be interesting to extend this morphism from the cohomology of a twisted Rota-Baxter operator to the cohomology of the corresponding NS-algebra introduced in this paper.
\end{remark}

\section{Appendix}\label{sec-appen}
In this appendix, we recall $L_\infty$-algebras and their Maurer-Cartan theory \cite{lada-markl,lada-stasheff,getzler}. We will follow the sign conventions of \cite{markl}.

Let $L = \bigoplus_i L_i$ be a graded vector space. A multilinear map $l : L^{\otimes k} \rightarrow L$ is said to be skew-symmetric if $l ( x_{\sigma (1)}, \ldots, x_{\sigma (k)} ) = (-1)^\sigma \epsilon (\sigma) l (x_1, \ldots, x_k)$, for $\sigma \in S_k$. Here $\epsilon (\sigma)$ is the Koszul sign.

\begin{defn}
A $L_\infty$-algebra consists of a graded vector space $L = \bigoplus_i L_i$ together with a collection $\{ l_k : L^{\otimes k} \rightarrow L |~ \mathrm{deg}(l_k) = 2-k \}_{k \geq 1}$ of skew-symmetric multilinear maps satisfying the following higher Jacobi identities: for each $n \geq 1$, we have
\begin{align*}
\sum_{i+j = n+1}^{} \sum_{\sigma}^{} (-1)^\sigma \epsilon (\sigma) ~ (-1)^{i (j-1)} ~ l_j   \big(  l_i (x_{\sigma (1)}, \ldots, x_{\sigma (i)}),  x_{\sigma (i+1)}, \ldots, x_{\sigma (n)} \big) = 0,
\end{align*}
where $\sigma$ runs over all $(i, n-i)$-unshuffles with $i \geq 1$.
\end{defn}

An element $\alpha \in L_1$ is called a Maurer-Cartan element in the $L_\infty$-algebra $(L, l_1, l_2, \ldots)$ if $\alpha $ satisfies
\begin{align*}
l_1 (\alpha ) + \frac{1}{2!} l_2 (\alpha, \alpha) - \frac{1}{3!} l_3 (\alpha, \alpha, \alpha ) - \cdots = 0.
\end{align*}

If $\alpha$ is a Maurer-Cartan element, then one can construct a new $L_\infty$-algebra twisted by $\alpha$ on the graded vector space $L$. In particular, if the given $L_\infty$-algebra $(L, l_1, l_2, \ldots)$ has $l_i =0$, for $i=1$ and $i > 3$, then the structure maps for the twisted $L_\infty$-algebra are given by
\begin{align*}
l_1' (x) = l_2 (\alpha, x) - \frac{1}{2}~ l_3 (\alpha, \alpha, x) ,~\quad
l_2' (x,y) =  l_2 (x, y) - l_3 (\alpha, x, y) ~~ \text{ and } ~~
l_3' (x, y, z) = l_3 (x, y, z).
\end{align*}

\vspace*{0.3cm}

\noindent {\bf Acknowledgements.} The research is supported by the fellowship of Indian Institute of Technology (IIT) Kanpur. The author thanks the Institute for support.

\end{document}